\documentclass{amsart}
\usepackage{amsmath,amsfonts,amssymb,stmaryrd}
\usepackage{graphicx,psfrag,subfigure}
\newtheorem{thm}{Theorem}
\newtheorem{rk}{Remark}
\newtheorem{prop}{Proposition}
\newtheorem{clly}{Corollary}
\newtheorem{lemma}{Lemma}
\newtheorem{defi}{Definition}
\newcommand{\pf}{{\flushleft{\bf Proof: }}}
\newcommand{\R}{{\mathbb{R}}}
\newcommand{\C}{{\mathbb{C}}}
\newcommand{\Z}{{\mathbb{Z}}}
\newcommand{\N}{{\mathbb{N}}}
\newcommand{\T}{{\mathbb{T}}}
\newcommand{\D}{{\mathbb{D}}}
\begin{document}
\title{Surface attractors.}

\author[J. Iglesias]{J. Iglesias}
\address{Universidad de La Rep\'ublica. Facultad de Ingenieria. IMERL.
Julio Herrera y Reissig 565. C.P. 11300. Montevideo, Uruguay}
\email{jorgei@fing.edu.uy }
\author[A. Portela]{A. Portela}
\address{Universidad de La Rep\'ublica. Facultad de Ingenieria. IMERL.
Julio Herrera y Reissig 565. C.P. 11300. Montevideo, Uruguay }
\email{aldo@fing.edu.uy }
\author[A. Rovella]{A. Rovella}
\address{Universidad de La Rep\'ublica. Facultad de Ciencias.
Centro de Matem\'atica. Igu\'a 4225. C.P. 11400. Montevideo,
Uruguay} \email{leva@cmat.edu.uy}
\author[J. Xavier]{J. Xavier}
\address{Universidad de La Rep\'ublica. Facultad de Ingenieria. IMERL.
Julio Herrera y Reissig 565. C.P. 11300. Montevideo, Uruguay }
\email{jxavier@fing.edu.uy }

\date{\today}
\begin{abstract}
Let $f$ be a continuous endomorphism of a surface $M$, and $A$ an
attracting set such that the restriction $f|_A: A \to A$ is a
$d:1$ covering map.
We show that if $f$ is a local homeomorphism in the
immediate basin $B^0_A$ of $A$, then $f$ is also a $d:1$ covering of $B^0_A$.
\end{abstract}
\maketitle
\begin{section}{Introduction}

Let $f$ be an endomorphism of a topological space $X$ and $U$ a nonempty
open proper subset of $X$. The pair $(f,U)$ is an {\em attracting pair}
of $X$, if the closure of $f(U)$ is contained in $U$.
By discarding the possibilities $U=\emptyset$ and $U=X$ we avoid trivial cases.
The {\em attracting set} associated to the attracting pair $(f,U)$ is defined as
$$
A=A(f,U)=\bigcap_{n\geq 0}f^n(U).
$$

Examples of attracting sets frequently appear.
By a theorem of C.Conley (\cite{con}), a homeomorphism which is not chain
recurrent defined in a compact space $X$, always has an attracting set.
Moreover, hyperbolic attractors are always attracting sets.

The basin of the attracting pair, also called the basin of the
attracting set $A$, is equal to $B=B_A=\cup_{n\geq 0}f^{-n}(U)$.
The immediate basin of $A$ is the union of the connected components of
$B$ that intersect $A$, and is denoted by $B^0_A$.

A point $p$ is a critical point of a continuous map $f$ if $f$ is
not local homeomorphism at $p$, meaning that there does not exist
a neighborhood $V$ of $p$ such that the restriction of $f$ to $V$
is a homeomorphism onto $f(V)$. The set of critical points of $f$
is denoted by $S_f$.

It is easy to prove that whenever $U$ is relatively compact, the
attracting set $A$ is compact, and if, in addition, $X$ is locally connected,
then $A$ has finitely many components.

We are interested in a kind of global problem.
Does there exist a preimage of $A$ contained in the immediate basin of $A$?
More precisely, we want to know if $\tilde A=f^{-1}(A)\setminus A$ may intersect $B^0_A$.

We will prove the following:

\begin{thm}
\label{intro1}
If $M$ is a compact, oriented, two dimensional manifold and $f$ is injective in $U$, then either $f$ is injective
in $B^0_A$ or there are critical points of $f$ in $B^0_A$.
\end{thm}

This result has important consequences regarding $C^1$ structural stability of endomorphisms.
Indeed, it is well known that a $C^1$ stable map must be Axiom A,
thus there exist transitive attracting sets. Moreover, as proved
by Przytycki \cite{prz}, the stability implies also that the map
is injective in every attractor $A$, so the above hypothesis are
fulfilled whenever $f$ is stable. In addition, as critical points
are forbidden for $C^1$ stability, one has the following
consequence.
\begin{clly}
If $f$ is a $C^1$ stable map of an oriented compact surface $M$, then its
restriction to the immediate basin of an attractor is injective.
\end{clly}

Other interesting conclusions can be derived from Theorem \ref{intro1};
for example, it can be proved that a $C^1$ stable map which is not a diffeomorphism nor it is expanding must have a saddle type basic piece
and a periodic attractor. These assertions are not trivial consequences, and we left their proofs to
forthcoming works.

Notice that theorem \ref{intro1} generalizes the corresponding results in one dimensional
manifolds. Indeed, the result has a trivial proof if $X$ is equal to the circle $S^1$. Moreover,
it is well known that for a rational map in the Riemann sphere every attractor is a periodic orbit
and the local inverse of $f$ defined in a neighborhood of the attractor can be extended until its domain
hits a critical point of $f$.

However, as will be shown in the final section, the result cannot be extended to manifolds of dimensions
three or more.

It is natural to ask if some similar conclusion as in Theorem \ref{intro1} can be deduced when the map $f$
is not one to one, but is a covering map from $U$ to $f(U)$. In this case, as will shown the examples on the
final section, it may happen that different points in $A$ may have different number of preimages in $A\cap U$,
so to extend the result above we need a restriction.

\begin{defi}
\label{normal} An attracting pair $(f,U)$ is called normal if
$f^{-1}(A)\setminus A$ is a closed set and $S_f\cap  U=\emptyset$.
\end{defi}
In particular, this holds whenever $f$ is injective in $U$.
We will prove in the next section that if $A$ is a connected
attracting set, then $(f,U)$ is normal if and only if  $f:A\to A$
is a covering map. The degree $d$ of this covering will be called
the degree of $A$.

\begin{thm}
\label{intro2}
Let $M$ be an oriented, compact surface, and  $(f,U)$ a normal attracting pair with associated attracting set $A$.\\
If $f$ restricted to $A$ is a covering of degree $d>1$ and $S_f\cap B^0_A=\emptyset$,
then the restriction of $f$ to the immediate basin of $A$ is also a
$d:1$ covering. In addition, the attracting set is an essential continuum in an annulus.
\end{thm}

Normal attractors have another interesting property: the complement of $A$ also has finitely many components.
It will be shown that if $A$ is a hyperbolic attractor of a $C^1$ map $f$ on a two dimensional manifold,
and the restriction of $f$ to $A$ is $d:1$ with $d$ greater than one, then the restriction of $f$ to the immediate basin of $A$
is conjugated to
$$
(z,y)\in S^1\times \R\to (z^d, y/2).
$$
Many properties of this map are well known (see \cite{bkru} and \cite{tsu}) and will be explained in the final section.
It follows that for a generic perturbation $f'$ of $f$ it holds that $(f',U)$ is an attracting pair that
fails to be normal, and the complement of the attracting set may have infinitely many components.
We will also show that the assertions of the theorem and the corollary are false in manifolds of dimension at
least three.

\end{section}

\begin{section}{Finiteness of components.}

This section is devoted to the statement of some general facts concerning normal attracting pairs.
The main goal is to prove the following:
\begin{prop}
\label{p1}
If $M$ is a compact oriented manifold without boundary, and $(f,U)$ is a normal attracting pair in $M$, then
$M\setminus A$ has finitely many components.
\end{prop}

\begin{defi}
\label{equivalent} Let $(f,U)$ be an attracting pair and assume
that $U'$ is an open subset of $X$. Then $(f,U')$ is equivalent to
$(f,U)$ if it is an attracting pair and defines the same
attracting set.
\end{defi}

Let $(f,U)$ be an attracting pair with associated
attracting set $A$. The set of connected components of a set $Y$
will be denoted by $\Pi_0(Y)$.\\

We begin by stating some general facts:
\begin{enumerate}
\item
The attracting set $A$ is closed, since $A=\cap f^n(U)=\cap
\overline{f^n(U)}$. Besides, $f(A)=A$.
\item
If $U'$ is a
neighborhood of $A$ such that $f^{ n+1}(U)\subset U'\subset f^{ n}(U)$ for some
positive integer $n$, then the pair $(f,U')$ is an attracting pair
equivalent to $(f,U)$. In particular, $(f,f(U))$ is equivalent to
$(f,U)$.
\item  Without loss of generality, one may suppose that every connected component of $U$ intersects $A$. Indeed,
one can remove unnecessary components and obtain an equivalent attracting pair.

\item
If $X$ is locally connected and the closure of $U$ is
compact, then $A$ has finitely many connected components.

Proof: The connected components of $U$ are open, intersect $A$ and
 are pairwise disjoint. Therefore, as $A$ is compact, $U$ has finitely many
connected components.  Besides, if $u \in \Pi_0 (U)$, then $f(u)$
intersects exactly one element of $\Pi_0 (U)$. Then,
$A=\cap_{n>0}f^n(U)$ has finitely many components (exactly as many
as $U$).

Finiteness of connected components of attractors was proved with a different set of hypothesis in {\cite[Theorem 1.4.6]{bue}}.

\end{enumerate}

\begin{lemma}
\label{covering} Let $X$ be a compact space and $(f,U')$ a normal
attracting pair in $X$ defining an attracting set $A$. If $A$ is
connected, then the restriction of $f$ to $A$ is a $d:1$ covering map onto $A$.
Moreover, there exists a neighborhood $U$ of $A$ such that
\begin{enumerate}
\item
\label{1}The pair $(f,U)$ is equivalent to $(f,U')$.
\item
\label{2}
The restriction of $f$ to $U$ is a $d:1$ covering
map onto $f(U)$.
\end{enumerate}
\end{lemma}
This assertion is trivial when $f$ is injective in $U'$.
\begin{proof}
Note that the restriction of $f$ to $A$ is a local homeomorphism. As $A$ is compact and connected, the first assertion follows immediately.



 By normality, there exists an open set $V\subset U$ such that $f^{-1}(A)\cap V =A$.  Moreover, there exists $V_0$, a neighborhood of $A$ contained in $V$, such that every point in
$V_0$ has exactly $d$ preimages in $V$. Let $n$ be such that
$f^n(U')$ is contained in $V_0$ and note that
$U=f^{-1}(f^{n+1}(U'))\cap V_0$ satisfies the assertions \ref{1} and
\ref{2}.
\end{proof}
We obtain the following corollary:
\begin{lemma}
\label{sufi}
If $X$ is compact, $A$ is connected and $S_f\cap U=\emptyset$, then the attracting pair $(f,U)$ is normal if and only if $f|_A$ is a covering onto $A$.
\end{lemma}
\begin{proof}
To prove the remaining part, take a sequence $(x_n)_{n\in \N} \subset f^{-1}(A)\backslash A$
and assume by contradiction that $x_n \to x \in A$. For each $n\geq 0$, $f(x_n)\in A$ has
$d$ preimages in $A$, denoted $y_n^1, \ldots, y_n^d$.  By passing to subsequences, assume
that each $\{y_n^i\}_{n\in \N}$ is convergent to a point $y^i\in A$.
Note that there exists $\epsilon >0$ such that $d(x,y)\geq \epsilon$
whenever $x,y$ are different preimages of the same $z\in A$.
Moreover, $d(y^i, y^j) \geq\epsilon $ for $i\neq j$ and $d(x, y^i)\geq\epsilon$ for all $i=1, \ldots, d$.  It
follows that $f(x)$ has $d+1$ preimages in $A$: this contradiction proves the assertion.
\end{proof}

Let $M$ be a compact manifold. We begin the proof of proposition \ref{p1} with a
simple topological fact.

\begin{lemma}
\label{l1}  Let $A_1$ and $A_2$ be closed disjoint subsets of $M$. Assume that
$\{d_n\}_{n>0}$ is a sequence of connected, pairwise disjoint sets in $M$
such that the boundary of each $d_n$ is contained in $A_1\cup A_2$
and intersects $A_1$. Then there exists $N$ such that for all
$n>N$ the boundary of $d_n$ is disjoint from $A_2$.
\end{lemma}

\begin{proof}
Let $\epsilon_0>0$ be such that for every $x\in M$, the
exponential map of $M$ is a diffeomorphism from the ball centered at $0$
and of radius $\epsilon_0$. Let $\epsilon<\epsilon_0$ be such that
the distance between $A_1$ and $A_2$ is greater than $2\epsilon$.
As $\{d_n\}$ is a disjoint sequence of open sets and $M$ is
compact, there exists $N>0$ such that $d_n$ does not contain a
ball of radius $\epsilon$ for every $n>N$. Take any $n>N$ and for
$i=1,2$ define $a^{i}_n=\{x\in d_n\ :\ B_M(x;\epsilon)\cap
A_i\neq\emptyset\}$, where $B_M(x;\epsilon)$ is the ball in $M$
image under the exponential of the corresponding ball in $T_xM$.
Note that $d_n=a^{1}_n\cup a^{2}_n$, that $a^{1}_n$ and $a^{2}_n$
are open and disjoint in $d_n$ and that $a^{1}_n$ is not empty. It
follows that $a^{2}_n$ is empty. As $n>N$ was arbitrary, the lemma
is proved.
\end{proof}

From now on it is assumed that $(f, U)$ is a normal attracting pair defined in a compact manifold $M$, where $U$ satisfies the
conditions of Lemma \ref{covering}.

Note that $f$ acts on $\Pi_0 (U\setminus A)$. More precisely,

\begin{enumerate}\item If $d\in \Pi_0 (U\setminus A)$, then $f(d)$
is contained in some $d'\in \Pi_0 (U\setminus A)$ because
$f^{-1}(A)\cap U = A$. Define $F: \Pi_0
(U\setminus A)\to  \Pi_0 (U\setminus A)$ by $F(d)=d'$.
\item\label{sobre} $F$ is surjective: if $d' \in \Pi_0
(U\setminus A)$ then $d'\cap f(U)\neq\emptyset$, because $f(U)$ is
a neighborhood of $A$ and the boundary of $d$ intersects $A$. It
follows that there exists $x\in U\setminus A$ such that $f(x)\in
d^{'}$, but every $x\in U\setminus A$ belongs to some $d\in \Pi_0
(U\setminus A)$, showing that $F(d)=d'$.
\end{enumerate}

\begin{lemma}
\label{l2} No component of
$M\setminus A$ is contained in $U$.
\end{lemma}
\begin{proof}
Assume by contradiction that $c\in \Pi_0(M\setminus A)$  is
contained in $U$. It follows that $f^n(c)$ is a component of
$M\setminus A$ for every $n>0$ because $f$ is an open mapping in $U$ and
$f^{-1}(A)\cap U = A$. This implies that $f^n(c)\in
\Pi_0(U\setminus A)$ for every $n\geq 0$. Note that $c$ is not
$F$-periodic, otherwise $c$ would be contained in $\cap_{n>0} f^n(U)$
which is a contradiction. Let $\{d_n\}_{n\in\Z}$ be a whole orbit
of $F$ such that $d_0=c$ (there exists such an orbit because $F$ is surjective).
 Moreover, as $c\cap A=\emptyset$, then there
exists $x\in c$ and $N\in \N$ such that $f^{-N}(x)\cap
U=\emptyset$; it follows that there exists a minimum $N_0>0$ such
that $d_{-N_{0}}$ is strictly contained in the component
$c_{-N_{0}}$ of $M\setminus A$ that contains $d_{-N_{0}}$.
Therefore, $d_{-n}$ is strictly contained in the component
$c_{-n}$ of $M\setminus A$ that contains $d_{-n}$, for every
$n\geq N_0$. So, the sequence $\{d_{n}\}_{-n>N_{0}}$ satisfies
the following properties:\\
It is a sequence of connected pairwise disjoint subsets of $M$,\\
the boundary of each $d_n$ is contained in $A\cup \partial U$,\\
the boundary of each $d_n$ intersects $A$, and\\
the boundary of each $d_n$ intersects $\partial U$ (because the
contrary assumption implies that the boundary of $d_n$ is
contained in $A$ and so $d_n$ is equal to a component of the
complement of $A$ in $M$).

As $A$ and $\partial U$ are closed disjoint sets, Lemma \ref{l1} implies the assertion.
\end{proof}

{\bf Proof of Proposition \ref{p1}.} The boundary of each
$d\in\Pi_0 (U\setminus A)$ intersects $\partial U$ by the previous lemma,
and it certainly intersects $A$.

By Lemma \ref{l1} it follows that $\Pi_0 (U\setminus A)$ is finite.
In particular, as each connected component of $M\backslash A$ intersects $U$, $M\backslash A$ has finitely many
connected components.\\

\end{section}

\begin{section}{Proof of Theorems.}

We devote this section to proving Theorem 1. We motivate the preliminary work with the following example. Suppose that $A$ is an attracting fixed point for a differentiable function $f$. We may take $U = B(A,r)$ for a suitable $r$ and obtain $\overline{f(U)}\subset U$.  In this case, $U$ retracts onto $f(U)$.  However, one may take a point $y\in U\backslash \overline{f(U)}$ and consider $U' = U\backslash \{y\}$ .  In this case $\overline{f(U')}\subset U'$, but  $U'$ does not retract onto $f(U')$. A fundamental part of the proof of Theorem 1 is finding a neighborhood $U$ of the attracting set such that  $U$ retracts onto $f(U)$ (see Lemma \ref{simple}).

\subsection{Genus and Nexus of a continuum}
Throughout this section, $M$ will stand for a compact oriented  connected surface without boundary.
A two dimensional submanifold with boundary $S$ is always diffeomorphic to the space obtained
by attaching a finite number $(g(S))$ of handles to a two sphere and then
removing a finite number $(\kappa(S))$ of open discs with disjoint closures and smooth boundaries.
The Euler characteristic of $S$, $\chi(S)=2-2g(S)-\kappa(S)$ is an invariant under homeomorphisms. If $f:S_1\to S_2$ is a covering map of degree $d$, then
$\chi(S_1)=d\chi(S_2)$.

The {\em genus} of $S$ is defined as the number of handles attached, and is denoted by $g(S)$. Equivalently, the genus of $S$ is
equal to the maximal number of simple closed disjoint curves one can delete from $S$ without
disconnecting it. In this form, the definition can be extended to the class of all connected
open subsets of $M$. Furthermore, the genus of a connected compact subset $K$ of $M$ is now defined
as the minimal genus of an open connected set that contains $K$. This definition
makes sense because whenever $W$ and $V$ are
open sets such that $V\subset W$, then $g(V)\leq g(W)$.
We denote by $\mathcal G(K)$ the class of all open neighborhoods of $K$ having minimal genus and whose closure is a submanifold of $M$.


The set $\mathcal G(K)$ can be characterized
as follows:
\begin{lemma}
\label{characterization}
If $S$ is an open neighborhood of $K$ whose closure is a submanifold with boundary, then $S\in \mathcal G(K)$ if and only if each simple closed curve in $S$
disconnects $S$ or intersects $K$.
\end{lemma}
\begin{proof}
Let $\gamma$ be a closed simple curve that does not disconnect $S$ neither intersects $K$. Then $S'=S\setminus \{\gamma\}$ is connected, is a neighborhood of $K$ and
has $g(S')=g(S)-1$; thus $S\notin \mathcal G(K)$.\\
If $S\notin \mathcal G(K)$, there exists a compact subset of $S$ whose interior $S'$ belongs to $\mathcal G(K)$. As the genus of $S$ is greater than the genus of $S'$, there
exists a simple closed curve $\gamma$ in $S\setminus S'$ that does not disconnect $S$, and does not intersect $K$.
\end{proof}

The {\em nexus} of a submanifold $S$ is defined as the number of connected components of its boundary. It can also be defined in terms of crosscuts. A crosscut is an arc in $S$
with extreme points in the boundary of $S$. The nexus of $S$ is $\kappa$ if $\kappa-1$ is the minimal number of disjoint crosscuts needed to connect the boundary of $S$.
We are interested in extending this definition to a continuum $K\subset M$.\\

\begin{lemma}
\label{biendef}
The number of connected components of $S\setminus K$ is the same for every $S\in\mathcal G(K)$.
\end{lemma}
\begin{proof}
Take $S_1$ and $S_2$ in $\mathcal G(K)$. The intersection of $S_1$
and $S_2$ is a neighborhood of $K$, but not necessarily a
submanifold, not necessarily connected. There exists $S'\in\mathcal G(K)$ such that
$S'$ is contained in the interior $ S_1\cap S_2$. If the number of components
of $S_1\setminus K$ is different to the number of components of
$S_2\setminus K$, then the number of one of them is different to
the number of components of $S'\setminus K$. So, we can assume
that $S^{'}$  is contained in the interior $ S_1$, and the number of connected
components of $S_1\setminus K$ is strictly smaller than the
number of connected components of $S^{'}\setminus K$ . So each
connected component of $S^{'}\setminus K$ is contained in a connected
component of $S_1\setminus K$. By assumption, there must
exists a connected component $c$ of $S_1\setminus K$ that contain at least two  connected component $d_1,d_2$ of $S^{'}\setminus K$.
 The boundary of $d_1$ contains a simple closed curve
$\gamma$ contained in $S_1$
that does not intersect $K$ and $S_0=S_1\setminus \gamma$
is connected. By Lemma \ref{characterization}, we arrive to a contradiction.\\
\end{proof}

\begin{defi}
\label{kappa}
Let $K$ be a compact and connected subset of $M$. Define the nexus of $K$ as the number of components of $S\setminus K$ where $S\in \mathcal G(K)$.
Denote by $\kappa(K)$ the nexus of $K$ .
\end{defi}

Now we will relate $\kappa(K)$ with $\kappa(S)$ for $S\in \mathcal G(K)$. Let us consider the class $\mathcal G_{\kappa} (K)$ of surfaces $S\in\mathcal G(K)$
having exactly one connected component of $\partial S$ in each connected component of $S\backslash K$.
It is clear that if $S\in \mathcal G_\kappa(K)$, then $\kappa (S) = \kappa (K)$.

We will often make use of the following

\begin{rk}\label{annuli}  If $S\in \mathcal G_{\kappa} (K) $, then $S\backslash K$ is a disjoint union of semi-open annuli.
\end{rk}
Indeed, $S\backslash K$ has genus zero as $S\in \mathcal G (K) $. Furthermore, $\kappa (S) = \kappa (K)$ gives us that each connected component of $S\backslash K$ has only two ends, therefore it is an annulus.

\begin{lemma}
\label{base}
If $K$ is a continuum such that $M\setminus K$ has finitely many components, then $\mathcal G_{\kappa} (K) $ is a basis of neighborhoods of $K$.
\end{lemma}
\begin{proof}
 It is enough to show that for any $S\in \mathcal G (K)$ there exists $S'\in \mathcal G_\kappa (K)$ such that $S'\subset S$. Note that as $M\setminus K$ has finitely many components, then $S\setminus K$ has finitely many components. Take an open disk $d_i$
in each connected component of $S\backslash K$ which does not contain any component of $\partial S$. Then define $S'_0 = S \backslash \cup _i d_i$.
For each connected component $C$ of $S'_0\backslash K$  having $k>1$ boundary components of $S$ (denoted $b_1,\ldots, b_k$) take
pairwise disjoint crosscuts $\gamma_j$ $0<j<k$ in $C$
connecting $b_j$ with $b_{j+1}$. Enlarge carefully the curves $\gamma_j $ to a small strip $c_j$ and define
$\tilde C=C\setminus (\cup_j c_j)$. Note that each $\tilde C$ is connected (each $b_i$ is a circle), contained
in $S'_0$ and has exactly one boundary component not intersecting $K$. Finally define $S'$ as the union of $K$
with the union of the $\tilde C$ for $C\in\Pi_0(S'_0\setminus K)$.
\end{proof}


We still need another topological fact.
\begin{lemma}
\label{l4}
Let $W\in \mathcal G_\kappa(K)$ and let $V\in \mathcal G(K)$, $V\subset W$.
Then there exists $S'\in\mathcal G_\kappa(K)$ satisfying the following properties:
\begin{enumerate}

\item $V\subset S'\subset W$,

\item $\partial S'\subset \partial V$

\item $\kappa(S')=\kappa(K)$
\end{enumerate}
\end{lemma}

\pf Clearly, $\kappa(V)\geq \kappa(K)$.  If $\kappa(V)=\kappa(K)$, we let $S' = V$.  Otherwise, $\kappa(V)> \kappa(K)$. This means that there exists $u\in \Pi_0 (W\backslash K)$ such that $V$ has at least two boundary components in $u$.  As $u$ is a semi-annulus (see Remark \ref{annuli}), only one of these boundary components is non-trivial in $\pi_1 (u, x_0)$. Let $\Gamma$ be the set of all $\gamma\in \Pi_0 (\partial V)$ that are trivial  in $\pi_1 (u, x_0)$.  Then, any $\gamma \in \Gamma$ is the boundary of a disk $D_\gamma \subset u$.  The surface $S' = V\cup_{\gamma \in \Gamma } D_\gamma$ satisfy all the conclusions of the lemma.


\subsection{An invariant $d:1$ continuum $K\neq M$, $d>1$, has nexus two.}

We turn now to dynamics. This subsection is devoted to the proof of the last assertion of Theorem \ref{intro2}.
In fact we prove a more general result, that applies also to invariant continua, not necessarily attracting.

The argument below rests on a simple fact: if $f:S_1\to S_2$
is a $d:1$ covering, $d>1$, and $S_2$ is an annulus, then $S_1$ is also an annulus.

\begin{thm}
\label{t2}
Let $K\subsetneq M$ be an $f$- invariant continuum such that: \begin{itemize}
\item $f|_K:K\to K$ is a $d:1$ covering map, $d>1$,
\item $f^{-1}(K)\backslash K$ is closed,
\item $M\backslash K$ has finitely many connected components. \end{itemize}
Then, $K$ has genus zero and nexus two.
\end{thm}
\begin{proof}
As $M\backslash K$ has finitely many connected components, by Lemma \ref{base}, there exists  $U\in \mathcal G_\kappa(K)$ such that $U\cap S_f = \emptyset$ and $U\cap (f^{-1}(K)\backslash K )= \emptyset$. Also there exists $V\in \mathcal G_\kappa(K)$
such that the connected component $V'$ of $f^{-1}(V)$ containing $K$ is contained in $U$.  Our choice of $U$ implies that $f|_{V'}: V' \to V $ is a $d:1$
covering map.
Moreover, as $K\subset V' \subset U$ one has $g(V') = g(V) = g$.  We begin by proving that $\kappa (V') = \kappa (V)$. As $V\in {\mathcal G}_{\kappa}(K)$, the connected component of $V\backslash K$  are topological semi-open annulus (see Remark \ref{annuli}).  As $V'\backslash K$ covers $V\backslash K$, it follows that the connected components of $V'\backslash K$ are also annuli. This implies that there is exactly one connected component of $\partial V'$ in each connected component of $U\backslash K$, that is $\kappa(V') = \kappa (V) = \kappa (U)$.

The fact that $f|_{V'}: V' \to V $ is a $d:1$ covering map gives $\chi (V') = d \chi (V)$, that is, $2-2g(V')-\kappa (V') = d(2-2g(V)-\kappa (V))$. But as $g(V)= g (V') $ and $\kappa (V') = \kappa (V)$, it comes that $2-2g-k = d (2-2g-k)$.  Moreover, $d\neq 1$ implies $\chi (V) = \chi (V') = 0$, showing that
$V\in {\mathcal G}_{\kappa}(K)$ is an annulus.
\end{proof}

\noindent
{\bf Proof of Theorem \ref{intro2}.}\\
We include an easy proof here inspired in the discussions above.  A different, unified proof of both
Theorems \ref{intro1} and \ref{intro2} is given in the next section.\\
Take an annular neighborhood $U$ of $A$ such that $f^{-1}(A)\cap U = A$ . The connected
component $U_1$ of $f^{-1}(U)$ containing $A$ is also an annulus
because $U\subset B^0_A$ and $B^0_A \cap S_f = \emptyset$. Moreover,
$f: U_1 \to U$ is $d:1$ by the choice of $U$. Define by induction $U_n$ as to
be the connected component of $f^{-1}(U_{n-1})$ containing $A$.
Note that $U_n$ is an annulus contained in $B_A^0$ for all $n\geq 1$.
Assuming that $f: U_n \to U_{n-1}$ is a $d:1$ covering it will be shown that
so is $f: U_{n+1} \to U_{n}$.
Let $c_1$ and $c_2$ be the connected component of $U_n\backslash A$. Each one is an
annulus one of whose boundary components is contained in $A$.
Let $c'_i$ be the connected component of $f^{-1}(c_i)$ that is contained in $U_{n+1}$.
Then $c'_i$ is an annulus, and $f: c'_i\to c_i$ a covering map.
It is claimed now that the component of the boundary of $c'_i$ intersecting $f^{-1}(A)$
is contained in $A$: otherwise, both connected component of the boundary of $c'_i$ are mapped by $f$ in $A$,
 which is impossible. So $f^{-1}(A)\cap (\overline{c_1^{'} \cup c_2^{'}})=A$, and as $f: U_{n+1} \to U_{n}$ is a covering map, then
such covering is $d:1$. As $\bigcup _n U_n = B_A^0$ the result follows.

\subsection{Unified proof.}

Let $M$ be a two dimensional manifold and $(f,U)$ be a normal attracting pair.
\begin{prop}
\label{p2}
There exists a neighborhood $U'$ of $A$ such that $(f,U')$ is equivalent to $(f,U)$, $f|_{U'}$ is a $d:1$ covering and $U'\in \mathcal G_\kappa(A)$.
\end{prop}
\begin{proof}
By Lemma \ref{base} there exists $W\in \mathcal G_\kappa(A)$, such that $\overline{ W}\subset U$. Note that every connected component of $U\setminus A$ contains exactly one connected component of  $W\setminus A$.
It follows that each component of $M\setminus W$ contains at least one component of $M\setminus U$ (otherwise $U$ would contain a component of $M\setminus A$,
that is not possible as stated in Lemma \ref{l2} of the previous section).

Then let $n>0$ be such that the closure of $f^n(U)$ is contained in $W$.
Choose some open set $V$ whose closure is a manifold with boundary and such that the closure of $f^{n+1}(U)$ is contained in $V$ and $V$ is contained in $f^{n}(U)$.
By Lemma \ref{l4}, there exists an open set $S'$ equal to the union of $V$ with those components of the complement of $V$ that are contained in $W$. \\
It will be proved that $\overline{f(S')}\subset S'$.
Note first that $\overline{f(V)}\subset V$ because $f^{n+1}(U)\subset V\subset f^n(U)$.\\
It is claimed now that if a component $d$ of the complement of $V$ is contained in $W$, then $f(d)\subset W$.
Denote by $b$ the boundary of $d$. As $f$ has no critical points in $U$, the boundary of $f(d)$ is contained in $f(b)$, that is contained in $V$,
and hence in $W$. Assume by contradiction that $f(d)$ is not contained in $W$. Then there exists a component of the boundary of $W$ contained in $f(d)$.
As each component of the complement
of $W$ contains at least one component of the complement of $U$,
it follows that $U$ does not contain $f(d)$. But this is a contradiction since $f(d)\subset f(U)\subset U$, hence $f(d)\subset W$, as claimed.\\
By definition of $S'$ it comes that the closure of $f(S')$ is contained in $S'$.

Now define $U'=f^{-1}(S')\cap U$. It is a covering because there are not critical points in $U$, and it is $d:1$ because its restriction to $A$ is $d:1$, and
$f^{-1}(A)\cap U\subset A$. Moreover $U'\in\mathcal G_\kappa(K)$ because $S'$ does, and $\overline{f(U')}\subset U'$ by the claim. Clearly the pair $(f,U')$ is
equivalent to $(f,U)$.
\end{proof}

\begin{clly}
\label{c3} Let $U\in\mathcal G_\kappa(A)$ where $A$ is the
attracting set defined by the attracting pair $(f,U)$ and $f:U\to f(U)$  a covering map. Then
$U\setminus f(U)$ is the union of a finite number of annuli.
\end{clly}
\begin{proof}
Note that $\overline{f(U)}$ is a submanifold, that the genus of $U$ and $f(U)$ coincide as both contain $K$, and that $U=(U\setminus f(U))\cup f(U)$. It follows that the genus of $U\setminus f(U)$ is equal to zero.
It is claimed now that $f(U)$ also belongs to $\mathcal G_\kappa(A)$. Using that the restriction of $f$ to $U$ is a covering, it follow that the number of components of the boundary of $f(U)$ is less than or equal to $\kappa(U)$. Now,
$f(U)\in\mathcal G(A)$ implies that $\kappa(U)\leq \kappa(f(U))$ Then, $\kappa(U)=\kappa(f(U))$, proving the claim.

It follows that both $U\backslash A$ and $f(U)\backslash A$ are disjoint union of annuli.  Moreover, as $f:U\to f(U)$  is a covering map, $\partial U$ is mapped onto $\partial f(U)$ and no component of $f(\partial U)$ can be homotopically trivial in $U$.  So, the connected components of $f(\partial U)$ are necessarily homotopic to connected components of $\partial U $, proving that $U\setminus f(U)$ is a finite union of annuli.
\end{proof}

It follows that $f(U)$ is a deformation retract of $U$. In
particular the attracting set is {\em simple}:

\begin{defi}
Let $(U,f)$ be an attracting pair with $f:U \to f(U)$ a $d:1$ covering map, $d\geq 1$, and consider $w\in f(U)$.  The pair
$(U,f)$ is {\it simple} if $i_*: \pi_1 (f(U),w) \to  \pi_1 (U,w) $ is onto, where $i_*$ is the map induced by the inclusion $i: (f(U),w) \to
(U,w)$.
\end{defi}

The following lemma, valid in any dimension, proves that simple attractors have the property we search for: being a $d:1$ covering in a neighborhood implies
the same property in the whole immediate basin.

\begin{lemma} \label{simple} If $(U,f)$ is a simple attracting pair with $f:U \to f(U)$ a $d:1$ covering map and $S_f\cap B_A^{0}=\emptyset$, then $f$ is $d:1$ in $B_A^{0}$.
\end{lemma}
\begin{proof} Let $w\in f(U)$ and $\{w_1, \ldots , w_d\}$ be the $d$ preimages of $w$ in $U$. We will show that for all $i=1, \ldots, d$ , $w_i$ has exactly $d$ preimages in $f^{-1}(U)\cap B_A^{0}$.

If $\beta$ is an arc joining $w$ and $w_i$, then the lift $\beta _j$ of $\beta$ starting at $w_j$ defines a preimage $x_j$ of $w_i$, $i,j \in {1, \ldots , d}$. Besides, as we assume that there are no critical points in $B_A^{0}$, the points  $x_j, j=1, \ldots, d$ are all different. We will show that for any other arc $\gamma $ joining $w$ and $w_i$ the lifts $\gamma _j$ starting at $w_j$
have their other extremity in $\{x_1, \ldots, x_d\}$,  $i,j \in {1, \ldots , d}$. Note that this implies the result. Indeed, if $x$ is a preimage of
$w_i$ in $f^{-1}(U)\cap B_A^{0}$, and $\delta$ is an arc joining $w_i$ and $x$, then $f(\delta)$ is an arc joining $w$ and $w_i$, and therefore the lift of
$f(\delta)$ starting at $w_i$ has its other extremity in $\{x_1, \ldots, x_d\}$, showing that $x\in \{x_1, \ldots, x_d\}$.

By hypothesis there exists a loop $\alpha \subset f(U)$ with basepoint at $w$ such that
 $[\beta \gamma ^{-1}]_{\pi_1 (U,w)} =
[\alpha]_{\pi_1 (U,w)} $. As $f:U \to f(U)$ is $d:1$, the lifts $\alpha _j$ of $\alpha$ starting at $w_j$ have their other extremities at $\{w_1, \ldots , w_d\}$, $j= 1, \ldots, d$.  Let $h_j: S^1\times [0,1] \to U$ be the lift of the homotopy joining $\alpha$ and $\beta\gamma
^{-1}$ starting at $\alpha _j$, $j= 1, \ldots, d$.
  Then, $h_j ((t,1))$ is the lift of $ \beta\gamma ^{-1}$ starting at $w_j$. So, $h_j
((t,1))= \beta _j \rho$, where $\rho$ is an arc starting at $x_j$ and having its other extremity $z\in \{w_1, \ldots , w_d\}$.  This implies that
$\rho^{-1}$ is the lift of $\gamma$ starting at $z$, showing that
the lifts of  $\gamma$ starting at $\{w_1, \ldots , w_d\}$
have their other extremities at $\{x_1, \ldots, x_d\}$.
\end{proof}
\end{section}

\begin{section}{Examples and Applications.}

\subsection{Homotopy of the immediate basin.}

Given $f:\T^{2}\to \T^{2}$, let $f_*:\pi_1(\T^{2},x_0)\to\pi_1(\T^{2},f(x_0))$ be the homomorphism of fundamental groups induced by $f$.\\

Note that $f_*$ can be seen as a $2\times 2$ matrix with integer coefficients.

\begin{prop}
Let $A$ be an injective attractor associated to the attracting pair $(f,U)$. If $\cap_{k\in\Z}f_*^{k}(\Z^{2})=\{0\}$, then every closed curve contained in $B^0_A$ is null homotopic in $\T^{2}$.
\end{prop}

Note that this does not say that $B^{0}_A$ is simply connected. For example, if $A$ is Plykin attractor, then $B_A^0$ is not simply
connected. Further, note that the hypothesis of the proposition is verified whenever $f_*$ is a hyperbolic matrix with determinant greater
than $1$.

\begin{proof}
Assume by contradiction that there exists a curve $\alpha$ in $B^{0}_A$ that is not null homotopic.
Now Theorem \ref{intro1} implies that $f':=f|_{B^{0}_{A}}$ is injective, so for every $k\in Z$, it holds that
$\alpha_k=f'^{-k}(\alpha)$ is a closed curve which is not null homotopic. Then $[\alpha ]\in \cap_{k\in \Z} f_*^{k}(\Z^2)$, a contradiction.
\end{proof}

\subsection{The hyperbolic normal attractor.} Let $f$ be a map of the cylinder $S^1\times \R$ given by
$$
f(z,y)=(z^d,\lambda y+\tau(z)),
$$
where $\lambda\in(0,1)$ and $\tau$ is a continuous function.

If $\tau=0$, then there exists a normal attractor of degree two, $A=S^1\times \{0\}$.
F.Przytycki (See \cite{prz}) used this map (with $\tau=0$) to show that $C^1$ $\Omega$-inverse stable
does not imply $C^1$ $\Omega$-stable.
The inverse limit is conjugated to a solenoid; M.Tsujii called it fat solenoidal attractor. He
gave a proof of the fact that for $\lambda$ close to $1$, and generic $\tau$ of class $C^2$, there exists a unique physical
measure that is absolutely continuous (see \cite{tsu}).
In \cite{bkru} some topological properties of this map were obtained: it has a global attractor that coincides
with the nonwandering set, and if $\tau$ is Lipschitz, then for $\lambda$ close to $1$ the attractor is an annulus.

These properties show the complexity of the dynamics of the family. We show here that this example is frequent.
If $\tau$ is of class $C^1$ then $f$ is $C^1$ and one can talk about hyperbolicity.
\begin{prop}
\label{p8}
Let $A$ be a normal and hyperbolic attractor in a two dimensional manifold $M$. If the degree of the restriction of $f$
to a neighborhood $U$ of $A$ is $d$, and $d>1$, then either $A$ is homeomorphic to a circle
and the restriction of $f$ to $A$ is conjugated to $z^d$, or $f$ is Anosov.
\end{prop}

\begin{proof}
It suffices to show that every (or one) unstable manifold is a circle. Let $p$ be a periodic point, assume it fixed,
and note that there exists an arc contained in its stable manifold that is a crosscut joining the connected components of the boundary
of the annular neighborhood $U$. As $d>1$, and $A$ hyperbolic, there exists a preimage $x$ of $p$, other
than $p$, that belongs to $A$. But the unstable manifold of $p$ is dense in $A$, so it must intersect the stable manifold of $x$,
say in a point $y$. This implies that the unstable manifold of $p$ is a curve that winds infinitely many in the annulus
(in other words, its lift to a band $\R\times [0,1]$ is unbounded) \\
It is claimed that $f(y)=p$ (so $y=x$).
Assume that $A$ has empty interior, which will be proved below . The assumption $f(y)\neq p$ implies that the unstable manifold
has infinitely many self intersections, and as the interior of $A$ is empty, a contradiction arrives because the complement of $A$
would have infinitely many components. This proves the claim. The claim implies that the unstable manifold of $p$ contains $f^{-n}(p)\cap A$
for every $n\geq 0$, and it follows that it must be a circle.
This, together with the next lemma, implies the result.
\end{proof}
\begin{lemma}
If a hyperbolic normal attractor has nonempty interior, then it is open.
\end{lemma}
\begin{proof}
Let $V$ be an open disc contained in $A$ and let $x\in A$ be a periodic point of $f$. Let $k$ be the period of $x$, and note that there exists a sequence $\{y_n\}_{n\geq 0}$ such that $f^k(y_{n+1})=y_n$ for every $n\geq 0$, $y_n\to x$ and $y_0\in V$. As $V\subset A$, and $A$ is a normal attractor, the connected component $V_j$ of $f^{-jk}(V)$ that contains $y_j$ is contained in $A$. But the $\lambda$-lemma implies that $V_j$ accumulates at the stable manifold of $x$ as $j$ goes to $\infty$. In conclusion, the stable manifold of $x$ is contained in $A$. It follows that every stable and unstable manifold of a point in $A$ is contained in $A$. This implies that $A$ is open by the local product structure.
\end{proof}

A first question is: Does there exist a hyperbolic attractor with nonempty interior that is not Anosov?\\
Other question: Does there exist a map in the family above whose attractor is normal and with nonempty interior?

\subsection{Not normal examples.}
1) Let $f:S^1\times \R\to S^1\times \R$, $f(z,y) = (z^d, \lambda y)$, $|\lambda|<1$.  Then,  $\overline {f(U)}\subset U$ if $U = S^1\times(-\epsilon, \epsilon)$ for a suitable $\epsilon >0$.  One can perturb $f$ to a map $f'$ such that $f'(U)$ looks like in Figure 1 (a) (see [Prz]).  The resulting attracting set $A$ is not normal.

2) The map $q(z)=z^2$ is a $2:1$ covering from $U'=\D\setminus\{0,1/2,-1/2\}$ onto $\D\setminus\{0, 1/4\}$. Let
$\phi$ a diffeomorphism defined in $\D$ that fixes $0$, sends $1/4$ to $1/2$ and whose image is a disc that avoids
the point $-1/2$. Note that $f_1=\phi q$ fixes $0$ and $1/2$ and carries $U'$ into itself. Moreover if the derivative
of $\phi$ at $1/2$ is adequately chosen, then $1/2$ will be a repelling fixed point for the composition $\phi q$. But $(f_1,U')$ is not an attracting pair because the closure of $f_1(U')$ is not contained in $U'$. To achieve an attracting pair with the same geometrical features as $f_1$, we proceed to modify the map $f_1$ around the origin and the open set $U'$ around the points $0$ and $1/2$.
The new map $f$ will coincide with $f_1$ for $|z|>1/3$, and with $q(2z/\epsilon)$ in $|z|<\epsilon$, whit $\epsilon <1/10$. Thus $f$ will have a fixed critical point at the origin and an expanding invariant curve at $|z|=\epsilon/2$.
Finally, define $U$ as $\D\setminus (D_0\cup D_{1/2}\cup D_{-1/2})$, where $D_0=D(0;\epsilon)$ and $D_{1/2}$ is a neighborhood of $1/2$
such that $f(D_{1/2})\supset\overline D_{1/2}$  and $D_{-1/2}$ is a neighborhood of $-1/2$
such that $f(D_{-1/2})\supset\overline D_{1/2}$.
Now $(f,U)$ is an attracting pair that is not normal because the restriction to the attractor is not a covering.

\begin{figure}[ht]
\centering
\psfrag{u}{$\small{U}$}\psfrag{fu}{$f^{'}(U)$}\psfrag{0}{$0$}
\psfrag{a}{$A$}\psfrag{b}{$b\in \mathcal B$}\psfrag{1}{$\pi$}
\psfrag{1}{$-\frac{1}{2}$}\psfrag{0}{$0$}\psfrag{2}{$\frac{1}{2}$}
\psfrag{f}{$f(U)$}
\subfigure[]{\includegraphics[scale=0.1]{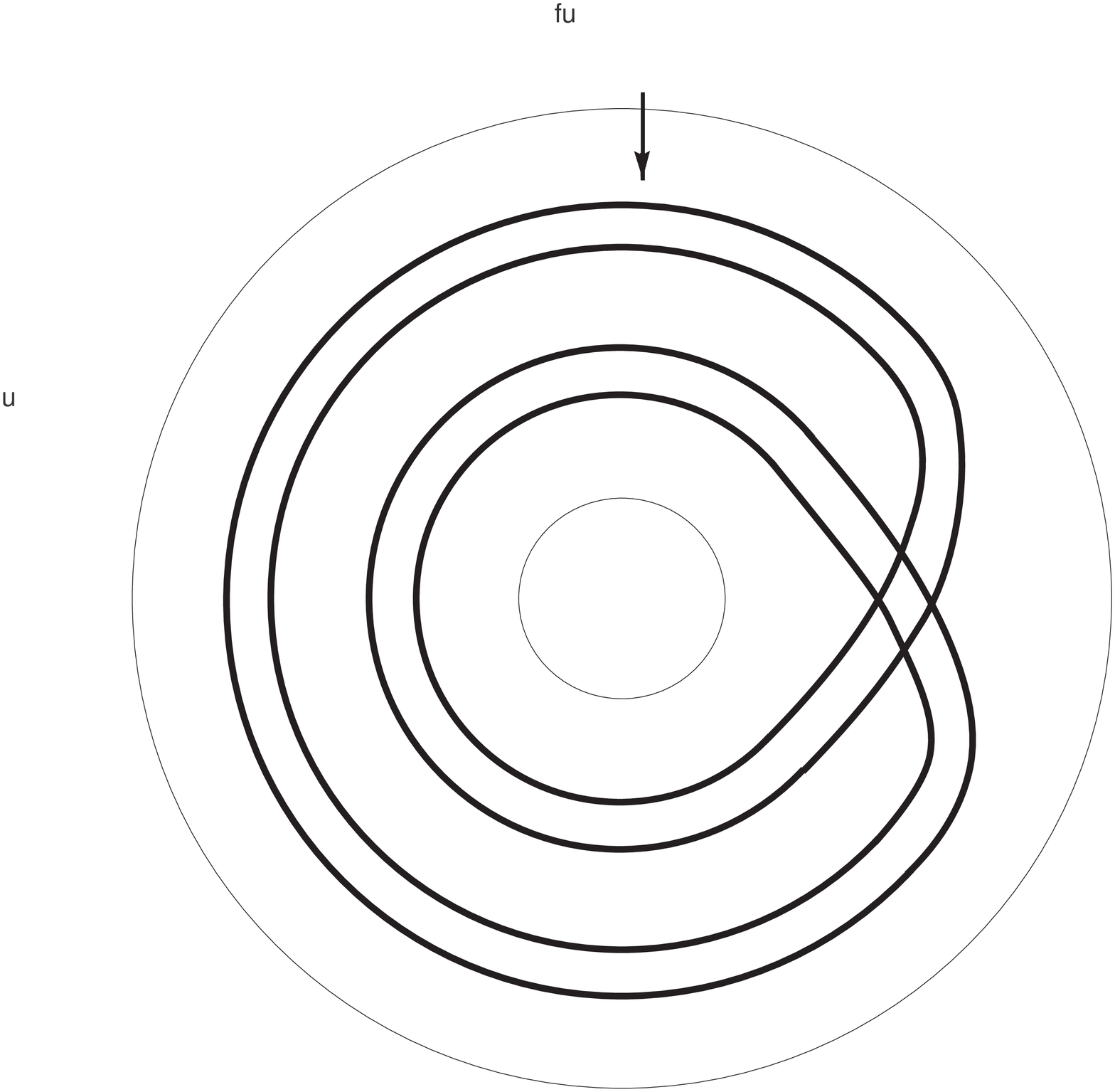}}
\subfigure[]{\includegraphics[scale=0.15]{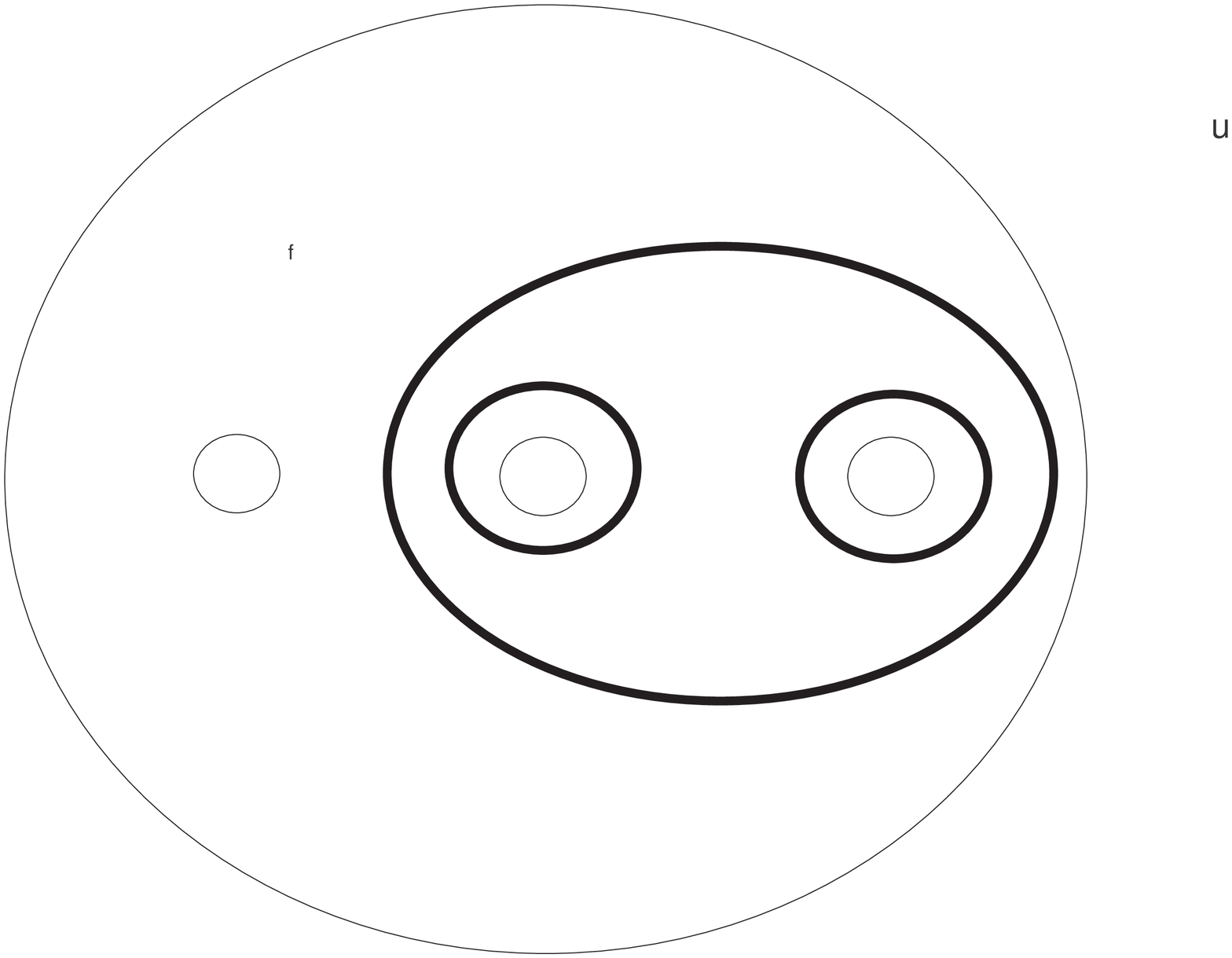}}

\caption{\label{noinyectivos}}
\end{figure}

\subsection{The counterexample.}
The results of Theorems \ref{intro1} and \ref{intro2} are not valid in manifolds of dimension greater than two.
The following example, introduced in \cite{ipr}, is a $C^1$ stable map that is not a diffeomorphism nor an expanding map.
The ambient manifold is $S^1\times S^2$, the two sphere seen as the compactification of the complex plane.  The map $f$ is given by
$f(z,w)=(z^2,(z+w)/3)$ for $w\in\C$, and $f(z,\infty)=(z^2,\infty)$. This map has degree two, it has no critical points, and its
non-wandering set is the union of a solenoid attractor and an expanding basic piece $S^1\times{\infty}$. The attractor $A$ is obtained
as the intersection of the forward images of the solid torus $S^1\times \D$, where $\D$ denotes the unit disc. It is easy to see
that $f$ is injective in $A$, and that the immediate basin coincides with the basin and is equal to $S^1\times \C$. As the map has
degree two, it follows that the restriction of the map to the immediate basin is not injective: the set $A'=f^{-1}(A)\setminus A$ is
not empty and contained in $B_A^0$.

\end{section}

\end{document}